\documentclass[12pt]{amsart}

\newtheorem{theorem}{Theorem}
\newtheorem{lemma}{Lemma}
\newtheorem{prop}{Proposition}
\newtheorem{cor}{Corollary}

\theoremstyle{definition}

\theoremstyle{remark}
\newtheorem{remark}{Remark}
\numberwithin{equation}{section}

\usepackage{amsmath,amssymb}

\def\re{\mathbb{R}}
\def\N{\mathbb{N}}

\def\({\left(}
\def\){\right)}

\def\eps{\varepsilon}
\def\la{\lambda}

\def\F{\mathcal{F}}

\def\S{\mathcal{S}}

\def\intO{\int_{\Omega}}

\def\weakto{\rightharpoonup}
\def\Pha{\Phi_{\alpha}}
\def\Psa{\Psi_{\alpha}}

\begin{document}
\title[]{Critical and subcritical fractional Trudinger-Moser type inequalities on $\mathbb{R}$}

\author{Futoshi Takahashi}
\address{Department of Mathematics, Graduate School of Science, Osaka City University, Sumiyoshi-ku, Osaka, 558-8585, Japan \\ 
Department of Mathematics, Osaka City University \& 
OCAMI, Sumiyoshi-ku, Osaka, 558-8585, Japan}
\email{futoshi@sci.osaka-cu.ac.jp}

\subjclass[2010]{Primary 35A23; Secondary 26D10.}

\keywords{Trudinger-Moser inequality, fractional Sobolev spaces, maximizing problem.}
\date{\today}

\dedicatory{}

\begin{abstract}
In this paper, we are concerned with the critical and subcritical Trudinger-Moser type inequalities 
for functions in a fractional Sobolev space $H^{1/2,2}$ on the whole real line. 
We prove the relation between two inequalities and discuss the attainability of the suprema.
\end{abstract}

\maketitle

%
%
\section{Introduction}

Let $\Omega \subset \re^N$, $N \ge 2$ be a domain with finite volume.
Then the Sobolev embedding theorem assures that $W^{1,N}_0(\Omega) \hookrightarrow L^q(\Omega)$ for any $q \in [1, +\infty)$,
however, a simple example shows that the embedding $W^{1,N}_0(\Omega) \hookrightarrow L^{\infty}(\Omega)$ does not hold.
Instead, functions in $W^{1,N}_0(\Omega)$ enjoy the exponential summability:
\[
	W^{1,N}_0(\Omega) \hookrightarrow \{ u \in L^N(\Omega) \, : \, \intO \exp \(\alpha |u|^{\frac{N}{N-1}} \) dx < \infty \quad \text{for any} \, \alpha > 0 \},
\]
see Yudovich \cite{Yudovich}, Pohozaev \cite{Pohozaev}, and Trudinger \cite{Trudinger}.
Later, Moser \cite{Moser} improved the embedding above as follows, now known as the Trudinger-Moser inequality:
\begin{align*}
	TM(\Omega, \alpha) = \sup_{u \in W^{1,N}_0(\Omega) \atop \| \nabla u \|_{L^N(\Omega)} \le 1} \frac{1}{|\Omega|} \intO \exp (\alpha |u|^{\frac{N}{N-1}}) dx
	\begin{cases}
	&< \infty, \quad \alpha \le \alpha_N, \\
	&= \infty, \quad \alpha > \alpha_N,
	\end{cases}
\end{align*}
here $\alpha_N = N \omega_{N-1}^{\frac{1}{N-1}}$ and $\omega_{N-1} = |S^{N-1}|$ denotes the area of the unit sphere in $\re^N$. 
On the attainability of $TM(\Omega, \alpha)$, Carleson-Chang \cite{Carleson-Chang}, Flucher \cite{Flucher}, and Lin \cite{KCLin} proved that 
$TM(\Omega, \alpha)$ is attained for any $0 < \alpha \le \alpha_N$.

On domains with infinite volume, for example on the whole space $\re^N$, the Trudinger-Moser inequality does not hold as it is.
However, several variants are known on the whole space.
In the following, let
\[
	\Phi_N(t) = e^t - \sum_{j=0}^{N-2} \frac{t^j}{j!} 
\]
denote the truncated exponential function.

First,
Ogawa \cite{Ogawa}, Ogawa-Ozawa \cite{Ogawa-Ozawa}, Cao \cite{Cao}, Ozawa \cite{Ozawa(JFA)}, and Adachi-Tanaka \cite{Adachi-Tanaka}
proved that the following inequality holds true, which we call Adachi-Tanaka type Trudinger-Moser inequality:
\begin{align*}
	A(N, \alpha) = \sup_{u \in W^{1,N}(\re^N) \setminus \{ 0 \} \atop \| \nabla u \|_{L^N(\re^N)} \le 1} \frac{1}{\| u \|^N_{L^N(\re^N)}} \int_{\re^N} \Phi_N (\alpha |u|^{\frac{N}{N-1}}) dx
	\begin{cases}
	&< \infty, \quad \alpha \, < \, \alpha_N, \\
	&= \infty, \quad \alpha \ge \alpha_N.
	\end{cases}
\end{align*}
The inequality enjoys the scale invariance under the scaling $u(x) \mapsto u_{\la}(x) = u(\la x)$ for $\la > 0$.
Note that the critical exponent $\alpha = \alpha_N$ is not allowed for the finiteness of the supremum.
Recently, it is proved that $A(N, \alpha)$ is attained for any $\alpha \in (0, \alpha_N)$ by 
Ishiwata-Nakamura-Wadade \cite{Ishiwata-Nakamura-Wadade} and Dong-Lu \cite{Dong-Lu}.
In this sense, Adachi-Tanaka type Trudinger-Moser inequality has a subcritical nature of the problem.

On the other hand, Ruf \cite{Ruf} and Li-Ruf \cite{Li-Ruf} proved that the following inequality holds true:
\begin{align*}
	B(N, \alpha) = \sup_{u \in W^{1,N}(\re^N) \atop \| u \|_{W^{1,N}(\re^N)} \le 1} \int_{\re^N} \Phi_N (\alpha |u|^{\frac{N}{N-1}}) dx
	\begin{cases}
	&< \infty, \quad \alpha \, \le \, \alpha_N, \\
	&= \infty, \quad \alpha > \alpha_N.
	\end{cases}
\end{align*}
Here $\| u \|_{W^{1,N}(\re^N)} = \( \| \nabla u \|_{L^N(\re^N)}^N + \| u \|_{L^N(\re^N)}^N \)^{1/N}$ is the full Sobolev norm.
Note that the scale invariance $(u \mapsto u_{\la})$ does not hold for this inequality.
Also note that the critical exponent $\alpha = \alpha_N$ is permitted to the finiteness. 

Concerning the attainability of $B(N, \alpha)$, the following facts have been proved:
\begin{itemize}
\item If $N \ge 3$, then $B(N, \alpha)$ is attained for $0 < \alpha \le \alpha_N$ \cite{Ruf}.
\item If $N = 2$, then  there exists $\alpha_* > 0$  such that $B(2, \alpha)$ is attained for $\alpha_* < \alpha \le \alpha_2 (= 4\pi)$ \cite{Ruf}, \cite{Ishiwata}.
\item If $N = 2$ and $\alpha >0$ is sufficiently small, then $B(2,\alpha)$ is not attained. \cite{Ishiwata}.
\end{itemize}
The non-attainability of $B(2,\alpha)$ for $\alpha$ sufficiently small is attributed to the non-compactness of ``vanishing" maximizing sequences, as described in \cite{Ishiwata}.

%
%
%
%
%
%
%
%
%

In the following, we focus our attention on the fractional Sobolev spaces.

Let $s \in (0,1)$, $p \in [1, +\infty)$ and let $\Omega \subset \re^N$ be a bounded Lipschitz domain.
For $s > 0$, let us consider the space
\[
	L_s(\re^N) = \left\{ u \in L^1_{loc}(\re^N) \, : \, \int_{\re^N} \frac{|u|}{1+|x|^{N+s}} dx < \infty \right\}.
\]
For $u \in L_s(\re^N)$, we define the fractional Laplacian $(-\Delta)^{s/2} u$ as follows:
First, for $\phi \in \S(\re^N)$, the rapidly decreasing functions on $\re^N$, 
$(-\Delta)^{s/2} \phi$ is defined via the normalized Fourier transform $\F$ as $(-\Delta)^{s/2} \phi (x) = \F^{-1}(|\xi|^s \F \phi (\xi))(x)$ for $x \in \re^N$. 
Then for $u \in L_s(\re^N)$, $(-\Delta)^{s/2} u$ is defined as the element of $\S^{\prime}(\re^N)$, the tempered distributions on $\re^N$, by the relation
\[
	\langle  \phi, (-\Delta)^{s/2} u \rangle = \langle (-\Delta)^{s/2} \phi, u \rangle = \int_{\re} (-\Delta)^{s/2} \phi \cdot u dx, \quad \phi \in \S(\re^N).
\]
Note that $L^p(\re^N) \subset L_s(\re^N)$ for any $p \ge 1$.
Also note that it could happen $supp((-\Delta)^{s/2} u) \not\subset \Omega$ even if $supp(u) \subset \Omega$ for some open set $\Omega$ in $\re^N$.

By using the above notion, we define {\it the Bessel potential space} $H^{s,p}(\Omega)$ for a (possibly unbounded) set $\Omega \subset \re^N$ as
\begin{align*}
	H^{s,p}(\re^N) &= \left\{ u \in L^p(\re^N) \, : \, (-\Delta)^{s/2} u \in L^p(\re^N) \right\}, \\
	\tilde{H}^{s,p}(\Omega) &= \left\{ u \in H^{s,p}(\re^N) \, : u \equiv 0 \quad \text{on} \, \re^N \setminus \Omega  \right\}.
\end{align*}

On the other hand, the Sobolev-Slobodeckij space $W^{s,p}(\re^N)$ is defined as
\begin{align*}
	W^{s,p}(\re^N) &= \left\{ u \in L^p(\re^N) \, : \, [u]_{W^{s,p}(\re^N)} < \infty \right\}, \\
	[u]_{W^{s,p}(\re^N)}^p &= \int_{\re^N} \int_{\re^N} \frac{|u(x)-u(y)|^p}{|x-y|^{N + sp}} dxdy,
\end{align*}
and for a bounded domain $\Omega \subset \re^N$, we define
\[
	\tilde{W}^{s,p}(\Omega) = \overline{C_c^{\infty}(\Omega)}^{\| \cdot \|_{W^{s,p}(\re^N)}}
\]
where $\| u \|_{W^{s,p}(\re^N)} = \( \| u \|_{L^p(\re^N)}^p + [u]_{W^{s,p}(\re^N)}^p \)^{1/p}$.
It is known that
\[
	\tilde{W}^{s,p}(\Omega) = \left\{ u \in W^{s,p}(\re^N) \, : u \equiv 0 \quad \text{on} \, \re^N \setminus \Omega  \right\}
\]
if $\Omega$ is a Lipschitz domain and 
$H^{s,p}(\re^N) = F^s_{p,2}(\re^N)$ (Triebel-Lizorkin space),
$W^{s,p}(\re^N) = B^s_{p,p}(\re^N)$ (Besov space).
Thus $H^{s,2}(\re^N) = W^{s,2}(\re^N)$, however in general, $H^{s,p}(\re^N) \ne W^{s,p}(\re^N)$ for $p \ne 2$.
See \cite{Parini-Ruf}, \cite{Iula} and the references therein.
%
%

Recently, Martinazzi \cite{Martinazzi} (see also \cite{Iula-Maalaoui-Martinazzi}) proved a fractional Trudinger-Moser type inequality on $\tilde{H}^{s,p}(\Omega)$ as follows:
Let $p \in (1, \infty)$ and $s = N/p$ for $N \in \N$.
Then for any open $\Omega \subset \re^N$ with $|\Omega| < \infty$, it holds
\begin{align*}
	\sup_{u \in \tilde{H}^{s,p}(\Omega) \atop \| (-\Delta)^{s/2} u \|_{L^p(\Omega)} \le 1} \frac{1}{|\Omega|} \intO \exp (\alpha |u|^{\frac{p}{p-1}}) dx
	\begin{cases}
	&< \infty, \quad \alpha \le \alpha_{N,p}, \\
	&= \infty, \quad \alpha > \alpha_{N,p}.
	\end{cases}
\end{align*}
Here $\alpha_{N,p} = \frac{N}{\omega_{N-1}} \( \frac{\Gamma((N-s)/2)}{\Gamma(s/2) 2^s \pi^{N/2}} \)^{-p/(p-1)}$.

We note that, differently from the classical case, the attainability of the supremum is not known even for $N = 1$ and $p=2$.

On the Sobolev-Slobodeckij spaces $\tilde{W}^{s,p}(\Omega)$ with $sp = N$, 
similar fractional Trudinger-Moser inequality is also proved by Parini-Ruf \cite{Parini-Ruf} when $N \ge 2$ and Iula \cite{Iula} when $N = 1$.
In this case, the result is weaker and the inequality holds true only for $0 \le \alpha < \alpha^{*}_{N,p}$ for some (explicit) value $\alpha^{*}_{N,p}$.
Also, it is not known the inequality holds or not when $\alpha = \alpha^{*}_{N,p}$. 

%
%

In the following, we are interested in the simplest one dimensional case, that is,
we put $N = 1$, $s = 1/2$ and $p = 2$.  
In this case, the Bessel potential space $H^{1/2,2}(\re)$ coincides with the Sobolev-Slobodeckij space $W^{1/2,2}(\re)$ and
both seminorms are related as
\begin{align*}
	 \| (-\Delta)^{1/4} u \|_{L^2(\re)}^2 = \frac{1}{2\pi} [u]_{W^{1/2,2}(\re)}^2,
\end{align*}
see Proposition 3.6. in \cite{Hitchhiker's guide}. 
Then the fractional Trudinger-Moser inequality in \cite{Martinazzi}, \cite{Iula-Maalaoui-Martinazzi} can be read as
\begin{prop}{(A fractional Trudinger-Moser inequality on $\tilde{H}^{1/2,2}(I)$)}
\label{Prop:fractional TM}
Let $I \subset \re$ be an open bounded interval. Then it holds
\begin{align*}
	\sup_{u \in \tilde{H}^{1/2,2}(I) \atop \| (-\Delta)^{1/4} u \|_{L^2(I)} \le 1} \frac{1}{|I|} \int_I e^{\alpha |u|^2} dx
	\begin{cases}
	&< \infty, \quad \alpha \le \alpha_{1,2} = \pi, \\
	&= \infty, \quad \alpha > \pi
	\end{cases}
\end{align*}
\end{prop}

For the fractional Adachi-Tanaka type Trudinger-Moser inequality on the whole line, put
\begin{equation}
\label{Aalpha}
	A(\alpha) = \sup_{u \in H^{1/2,2}(\re) \setminus \{0\} \atop \| (-\Delta)^{1/4} u \|_{L^2(\re)} \le 1} \frac{1}{\| u \|_{L^2(\re)}^2} \int_{\re} \( e^{\alpha u^2} - 1 \) dx.
\end{equation}
Then by the precedent results by Ogawa-Ozawa \cite{Ogawa-Ozawa} and Ozawa \cite{Ozawa(JFA)},
it is known that $A(\alpha) < \infty$ for small exponent $\alpha$.

On the other hand, it is already known a fractional Li-Ruf type Trudinger-Moser inequality on $H^{1/2,2}(\re)$:
\begin{prop}{(Iula-Maalaoui-Martinazzi \cite{Iula-Maalaoui-Martinazzi})}
\label{Prop:fractional Li-Ruf}
\begin{equation}
\label{Balpha}
	B(\alpha) = \sup_{u \in H^{1/2,2}(\re) \atop \| u \|_{H^{1/2,2}(\re)} \le 1} \int_{\re} \( e^{\alpha u^2} - 1 \) dx
	\begin{cases}
	&< \infty, \quad \alpha \, \le \, \pi, \\
	&= \infty, \quad \alpha > \pi.
	\end{cases}
\end{equation}
Here
\[
	\| u \|_{H^{1/2,2}(\re)} = \( \| (-\Delta)^{1/4} u \|_{L^2(\re)}^2 + \| u \|_{L^2(\re)}^2 \)^{1/2}
\]
is the full Sobolev norm on $H^{1/2,2}(\re)$.
\end{prop}

Concerning $A(\alpha)$ in (\ref{Aalpha}), a natural question is that to what range of the exponent $\alpha$ the supremum is finite.
As pointed out in \cite{Iannizzotto-Squassina}, it remained an open problem for a while.
In this paper, first we prove the finiteness of supremum in the full range of values of exponent.

\begin{theorem}{(Full range Adachi-Tanaka type on $H^{1/2,2}(\re)$)}
\label{Theorem:fractional AT}
We have
\[
	A(\alpha) = \sup_{u \in H^{1/2,2}(\re) \setminus \{0\} \atop \| (-\Delta)^{1/4} u \|_{L^2(\re)} \le 1} \frac{1}{\| u \|_{L^2(\re)}^2} \int_{\re} \( e^{\alpha u^2} - 1 \) dx.
	\begin{cases}
	&< \infty, \quad \alpha \, < \, \pi, \\
	&= \infty, \quad \alpha \ge \pi.
	\end{cases}
\]
\end{theorem}

Ozawa \cite{Ozawa} proved that the Adachi-Tanaka type Trudinger-Moser inequality is equivalent to the Gagliardo-Nirenberg type inequality, 
and he also proved an exact relation between the best constants of both inequalities.
As a result, we have the next corollary.
\begin{cor}
Set
\[
	\beta_0 = \limsup_{q \to \infty} \sup_{u \in H^{1/2,2}(\re), u \ne 0} \frac{\| u \|_{L^q(\re)}}{q^{1/2} \| (-\Delta)^{1/4} u \|_{L^2(\re)}^{1-2/q} \| u \|_{L^2(\re)}^{2/q}}.
\] Then $\beta_0 = (2\pi e)^{-1/2}$.
\end{cor}

Furthermore, we obtain the relation between the suprema of both critical and subcritical Trudinger-Moser type inequalities
along the line of Lam-Lu-Zhang \cite{Lam-Lu-Zhang}.
\begin{theorem}{(Relation)}
We have
\label{Theorem:relation}
\[
	B(\pi) = \sup_{\alpha \in (0, \pi)} \frac{1 - \(\alpha / \pi \)}{\(\alpha / \pi\)} A(\alpha).
\]
\end{theorem}

Also we obtain how Adachi-Tanaka type supremum $A(\alpha)$ behaves when $\alpha$ tends to $\pi$.

\begin{theorem}{(Asymptotic behavior)}
\label{Theorem:asymptotic}
There exist $C_1, C_2 > 0$ such that for any $\alpha < \pi$ which is close enough to $\pi$, 
it holds
\[
	\frac{C_1}{1 - \alpha / \pi} \le A(\alpha) \le \frac{C_2}{1 - \alpha / \pi}.
\]
\end{theorem}
Note that the estimate from the above follows from Theorem \ref{Theorem:relation} and Proposition \ref{Prop:fractional Li-Ruf}.
On the other hand, we will see that that the estimate from the below follows from a computation using the Moser sequence.

Concerning the existence of maximizers of Adachi-Tanaka type supremum $A(\alpha)$ in (\ref{Aalpha}), we see
\begin{theorem}{(Attainability of $A(\alpha)$)}
\label{Theorem:Attainability}
$A(\alpha)$ is attained for any $\alpha \in (0, \pi)$.
\end{theorem}

On the other hand, as for $B(\alpha)$ in (\ref{Balpha}), we have
\begin{theorem}{(Non-attainability of $B(\alpha)$)}
\label{Theorem:Non-attainability}
For $0 < \alpha << 1$, $B(\alpha)$ is not attained. 
\end{theorem}

It is plausible that there exists $\alpha_* > 0$ such that $B(\alpha)$ is attained for $\alpha_* < \alpha \le \pi$,
but we do not have a proof up to now.

Finally, we improve the subcritical Adachi-Tanaka type inequality along the line of Dong-Lu \cite{Dong-Lu}:

\begin{theorem}
\label{Theorem:Dong-Lu}
For $\alpha > 0$, set
\begin{equation}
\label{Ealpha}
	E(\alpha) = \sup_{u \in H^{1/2,2}(\re) \setminus \{0\} \atop \| (-\Delta)^{1/4} u \|_{L^2(\re)} \le 1} \frac{1}{\| u \|_{L^2(\re)}^2} \int_{\re} u^2 e^{\alpha u^2} dx.
\end{equation}
Then we have
\[
	E(\alpha) 
	\begin{cases}
	&< \infty, \quad \alpha \, < \, \pi, \\
	&= \infty, \quad \alpha \ge \pi.
	\end{cases}
\]
Furthermore, $E(\alpha)$ is attained for all $\alpha \in (0, \pi)$.
\end{theorem}
Since $e^{\alpha t^2} - 1 \le \alpha t^2 e^{\alpha t^2}$ for $t \in \re$, Theorem \ref{Theorem:Dong-Lu} extends Theorem \ref{Theorem:fractional AT}. 
In the classical case, Dong-Lu used a rearrangement technique to reduce the problem to one-dimension and obtained the similar inequality by estimating a one-dimensional integral.
The method is similar to \cite{Carleson-Chang}.
In the fractional setting $H^{1/2,2}$, we cannot follow this argument and we need a new idea.

The organization of the paper is as follows:
In  section 2, we prove Theorem \ref{Theorem:fractional AT}, \ref{Theorem:relation}, and \ref{Theorem:asymptotic}.
In  section 3, we prove Theorem \ref{Theorem:Attainability} and \ref{Theorem:Non-attainability}.
In  section 4, we prove Theorem \ref{Theorem:Dong-Lu}.

%
%

\section{Proof of Theorem \ref{Theorem:fractional AT}, \ref{Theorem:relation}, and \ref{Theorem:asymptotic}}

For the proofs of Theorem \ref{Theorem:fractional AT}, \ref{Theorem:relation}, and \ref{Theorem:asymptotic}, we prepare several lemmas.
\begin{lemma}
\label{Lemma1}
Set
\begin{equation}
\label{A_tilde}
	\tilde{A}(\alpha) 
	= \sup_{{u \in H^{1/2,2}(\re) \setminus \{ 0 \} \atop \| (-\Delta)^{1/4} u \|_{L^2(\re)} \le 1} \atop \| u \|_{L^2(\re)} = 1} 
	\int_{\re} \( e^{\alpha u^2} - 1 \) dx.
\end{equation}
Then $\tilde{A}(\alpha) = A(\alpha)$ for any $\alpha > 0$. 
\end{lemma}

\begin{proof}
For any	$u \in H^{1/2,2}(\re) \setminus \{ 0 \}$ and $\la >0$, we put $u_{\la}(x) = u(\la x)$ for $x \in \re$.
Then we have
\begin{equation}
\label{scaling}
	\begin{cases}
	&\| (-\Delta)^{1/4} u_{\la} \|_{L^2(\re)} = \| (-\Delta)^{1/4} u \|_{L^2(\re)}, \\ 
	&\| u_{\la} \|^2_{L^2(\re)} = \la^{-1} \| u \|^2_{L^2(\re)},
	\end{cases}
\end{equation}
since 
\begin{align*}
	2\pi \| (-\Delta)^{1/4} u_{\la} \|_{L^2(\re)}^2 &= [u_{\la}]_{W^{1/2,2}(\re)}^2 \\
	&= \int_{\re} \int_{\re} \frac{|u(\la x)-u(\la y)|^2}{|x-y|^2} dxdy \\
	&= \int_{\re} \int_{\re} \frac{|u(\la x)-u(\la y)|^2}{|\la x- \la y|^2} d(\la x)d(\la y) \\
	&= [u]_{W^{1/2,2}(\re)}^2 = 2\pi \| (-\Delta)^{1/4} u \|_{L^2(\re)}^2.
\end{align*}
Thus for any $u \in H^{1/2,2}(\re) \setminus \{ 0 \}$ with $\| (-\Delta)^{1/4} u \|_{L^2(\re)} \le 1$,
if we choose $\la = \| u \|^2_{L^2(\re)}$, then $u_{\la} \in H^{1/2,2}(\re)$ satisfies 
\[
	\| (-\Delta)^{1/4} u_{\la} \|_{L^2(\re)} \le 1  \quad \text{and} \quad \| u_{\la} \|^2_{L^2(\re)} = 1.
\]
Thus 
\[
	\frac{1}{\| u \|_{L^2(\re)}^2} \int_{\re} \( e^{\alpha u^2} - 1 \) dx = \int_{\re} \( e^{\alpha u_{\la}^2} - 1 \) dx 
	\le \tilde{A}(\alpha),
\] 
which implies $A(\alpha) \le \tilde{A}(\alpha)$.
The opposite inequality is trivial.
\end{proof}

\begin{lemma}
\label{Lemma2}
For any $0 < \alpha < \pi$, it holds
\[
	A(\alpha) \le \frac{\(\alpha / \pi\)}{1 - \(\alpha / \pi \)} B(\pi).
\]
\end{lemma}

\begin{proof}
Choose any $u \in H^{1/2,2}(\re)$ with $\| (-\Delta)^{1/4} u \|_{L^2(\re)} \le 1$ and  $\| u \|_{L^2(\re)} = 1$. 
Put $v(x) = C u(\la x)$ where $C^2 = \alpha/\pi \in (0,1)$ and $\la = \frac{C^2}{1 - C^2}$.
Then by scaling rules (\ref{scaling}), we see
\begin{align*}
	\| v \|^2_{H^{1/2,2}(\re)} &= \| (-\Delta)^{1/4} v \|^2_{L^2(\re)} + \| v \|^2_{L^2(\re)} \\
	&= C^2 \| (-\Delta)^{1/4} u \|^2_{L^2(\re)} + \la^{-1} C^2 \| u \|^2_{L^2(\re)} \\
	&\le C^2 + \la^{-1} C^2 = 1.
\end{align*}
Also we have
\begin{align*}
	\int_{\re} \( e^{\pi v^2} - 1 \) dx &= \int_{\re} \( e^{\pi C^2 u^2(\la x)} - 1 \) dx \\ 
	&= \la^{-1} \int_{\re} \( e^{\pi C^2 u^2(y)} - 1 \) dy \\ 
	&= \frac{1 - C^2}{C^2} \int_{\re} \( e^{\alpha u^2(y)} - 1 \) dy \\ 
	&= \frac{1 - \(\alpha / \pi \)}{\(\alpha / \pi\)} \int_{\re} \( e^{\alpha u^2(y)} - 1 \) dy. 
\end{align*}
Thus testing $B(\pi)$ by $v$, we see
\[
	B(\pi) \ge \int_{\re} \( e^{\pi v^2} - 1 \) dx \ge \frac{1 - \(\alpha / \pi \)}{\(\alpha / \pi\)} \int_{\re} \( e^{\alpha u^2(y)} - 1 \) dy. 
\]
By taking the supremum for $u \in H^{1/2,2}(\re)$ with $\| (-\Delta)^{1/4} u \|_{L^2(\re)} \le 1$ and  $\| u \|_{L^2(\re)} = 1$,
we have
\[
	B(\pi) \ge \frac{1 - \(\alpha / \pi \)}{\(\alpha / \pi\)} \tilde{A}(\alpha).
\]
Finally, Lemma \ref{Lemma1} implies the result.
\end{proof}

\vspace{1em}\noindent
{\it Proof of Theorem \ref{Theorem:fractional AT}}: 
The assertion that $A(\alpha) < \infty$ for $\alpha < \pi$ follows from Lemma \ref{Lemma2} and the fact $B(\pi) < \infty$ by Proposition \ref{Prop:fractional Li-Ruf}.

For the proof of $A(\pi) = \infty$, we use the Moser sequence 
\begin{align}
\label{Moser sequence}
	&u_{\eps} = 
	\begin{cases}
	\( \log  (1/\eps) \)^{1/2}, &\quad \text{if} \, |x| < \eps, \\
	\frac{\log (1/|x|)}{\( \log (1/\eps) \)^{1/2}}, &\quad \text{if} \, \eps < |x| < 1, \\
	0, &\quad \text{if} \, 1 \le |x|,
	\end{cases}
\end{align}
and its estimates
\begin{align}
\label{Moser_estimates(1)}
	&\| (-\Delta)^{1/4} u_{\eps} \|^2_{L^2(\re)} = \pi + o(1), \\ 
\label{Moser_estimates(2)}
	&\| (-\Delta)^{1/4} u_{\eps} \|^2_{L^2(\re)} \le \pi \( 1 + (C \log (1/\eps) )^{-1} \), \\
\label{Moser_estimates(3)}
	&\| u_{\eps} \|^2_{L^2(\re)} = O\( \( \log (1/\eps) \)^{-1} \)
\end{align}
as $\eps \to 0$ for some $C > 0$.
Note $u_{\eps} \in \tilde{W}^{1/2,2}((-1,1)) \subset W^{1/2,2}(\re) = H^{1/2,2}(\re)$.
For the estimate (\ref{Moser_estimates(1)}), we refer to Iula \cite{Iula} Proposition 2.2.
For the estimate (\ref{Moser_estimates(2)}), we refer to \cite{Iula} equation (35).
Actually, after a careful look of the proof of Proposition 2.2 in \cite{Iula}, we confirm that 
\[
	\lim_{\eps \to 0} \( \log (1/\eps) \) \( \| (-\Delta)^{1/4} u_{\eps} \|^2_{L^2(\re)} - \pi \)  \le C
\]
for a positive $C > 0$, which implies (\ref{Moser_estimates(2)}).
For (\ref{Moser_estimates(3)}), we compute
\begin{align*}
	\| u_{\eps} \|_{L^2(\re)}^2 &= \int_{|x| \le \eps} \(\log (1/\eps) \) dx 
	+ \int_{\eps < |x| \le 1} \( \frac{\log (1/|x|)}{\( \log (1/\eps) \)^{1/2}} \)^2 dx \\
	&= 2\eps \log (1/\eps) + \frac{2}{\log (1/\eps)} \int_{\log (1/\eps)}^0 t^2 (-e^t) dx \\ 
	&= 2\eps \log (1/\eps) + \frac{2}{\log (1/\eps)} \( \Gamma(3) + o(1) \)
\end{align*}
as $\eps \to 0$.
Thus we obtain (\ref{Moser_estimates(3)}).

By testing $A(\pi)$ by $v_{\eps} = u_{\eps} /\| (-\Delta)^{1/4} u_{\eps} \|_{L^2(\re)}$, we have
\begin{align*}
	A(\pi) &\ge \frac{1}{\| v_{\eps} \|_{L^2(\re)}^2} \int_{\re} \( e^{\pi v_{\eps}^2} - 1 \) dx \\
	&\ge \frac{\| (-\Delta)^{1/4} u_{\eps} \|^2_{L^2(\re)}}{\| u_{\eps} \|_{L^2(\re)}^2} \int_{|x| \le \eps} \( e^{\pi v_{\eps}^2} - 1 \) dx \\
	&\ge \frac{\| (-\Delta)^{1/4} u_{\eps} \|^2_{L^2(\re)}}{\| u_{\eps} \|_{L^2(\re)}^2} \eps \exp \( \pi  \frac{\log (1/\eps)}{\| (-\Delta)^{1/4} u_{\eps} \|^2_{L^2(\re)}} \) \\
	&\ge \frac{\| (-\Delta)^{1/4} u_{\eps} \|^2_{L^2(\re)}}{\| u_{\eps} \|_{L^2(\re)}^2} 
	\eps \exp \( \frac{\log (1/\eps)}{1 + (C \log (1/\eps) )^{-1}} \)
\end{align*}
since $e^t -1 \ge (1/2) e^t$ for $t$ large and (\ref{Moser_estimates(2)}).
Also since
\[
	\frac{t}{1 + \frac{1}{Ct}} - t = \frac{-1/C}{1 + \frac{1}{Ct}} \to -\frac{1}{C} \quad \text{as} \, t \to \infty,
\]
we see $\frac{t}{1 + \frac{1}{Ct}} = t - 1/C + o(1)$ as $t \to \infty$.
Put $t = \log (1/\eps)$, we see
\begin{align*}
	\exp \( \frac{\log (1/\eps)}{1 + (C \log (1/\eps) )^{-1}} \) = \exp \( \log (1/\eps) - 1/C +o(1) \) =  (1/\eps )e^{-1/C + o(1)}, 
\end{align*}
which leads to
\[
	\eps \exp \( \frac{\log (1/\eps)}{1 + (C \log (1/\eps) )^{-1}} \) \ge e^{-1/C + o(1)} \ge \delta > 0
\]
for some $\delta > 0$ independent of $\eps \to 0$.
Therefore, by (\ref{Moser_estimates(1)}), (\ref{Moser_estimates(2)}), (\ref{Moser_estimates(3)}), we have for $\delta' > 0$ 
\begin{align*}
	A(\pi) \ge \frac{\pi + o(1)}{(C \log (1/\eps)))^{-1}} \delta \ge \delta' \( \log (1/\eps) \) \to \infty
\end{align*}
as $\eps \to 0$.
This proves $A(\pi) = \infty$.
\qed

\vspace{1em}\noindent
{\it Proof of Theorem \ref{Theorem:relation}}: 
By Lemma \ref{Lemma2}, we have
\[
	B(\pi) \ge \sup_{\alpha \in (0, \pi)} \frac{1 - \(\alpha / \pi \)}{\(\alpha / \pi\)} A(\alpha).
\]
Let us prove the opposite inequality.
Let $\{ u_n \} \subset H^{1/2,2}(\re)$, $u_n \ne 0$, $\| (-\Delta)^{1/4} u_n \|^2_{L^2(\re)} + \| u_n \|^2_{L^2(\re)} \le 1$,
be a maximizing sequence of $B(\pi)$.
We may assume $\| (-\Delta)^{1/4} u_n \|^2_{L^2(\re)} < 1$ for any $n \in \N$.
Put
\[
	\begin{cases}
	&v_n(x) = \frac{u_n(\la_n x)}{\| (-\Delta)^{1/4} u_n \|_{L^2(\re)}}, \quad (x \in \re) \\
	&\la_n = \frac{1 - \| (-\Delta)^{1/4} u_n \|^2_{L^2(\re)}}{\| (-\Delta)^{1/4} u_n \|^2_{L^2(\re)}} > 0.
	\end{cases}
\]
Thus by (\ref{scaling}), we see
\begin{align*}
	&\| (-\Delta)^{1/4} v_n \|^2_{L^2(\re)} = 1, \\
	&\|  v_n \|^2_{L^2(\re)} = \frac{\la_n^{-1}}{\| (-\Delta)^{1/4} u_n \|^2_{L^2(\re)}} \| u_n \|^2_{L^2(\re)}
	= \frac {\| u_n \|^2_{L^2(\re)}}{1 - \| (-\Delta)^{1/4} u_n \|^2_{L^2(\re)}} \le 1,
\end{align*}
since $\| (-\Delta)^{1/4} u_n \|^2_{L^2(\re)} + \| u_n \|^2_{L^2(\re)} \le 1$.
Thus, setting  $\alpha_n = \pi \| (-\Delta)^{1/4} u_n \|^2_{L^2(\re)} < \pi$ for any $n \in \N$,
we may test $A(\alpha_n)$ by $\{ v_n \}$, which results in
\begin{align*}
	B(\pi) + o(1) &= \int_{\re} \( e^{\pi u_n^2(y)} - 1 \) dy \\ 
	&= \la_n \int_{\re} \( e^{\pi \| (-\Delta)^{1/4} u_n \|^2_{L^2(\re)} v_n^2(x)} - 1 \) dx \\ 
	&\le \la_n \frac{1}{\| v_n \|^2_{L^2(\re)}} \int_{\re} \( e^{\alpha_n  v_n^2(x)} - 1 \) dx \\ 
	&\le \la_n A(\alpha_n) = \frac{1 - \(\alpha_n / \pi \)}{\(\alpha_n / \pi\)} A(\alpha_n) \\
	&\le \sup_{\alpha \in (0, \pi)} \frac{1 - \(\alpha / \pi \)}{\(\alpha / \pi\)} A(\alpha).
\end{align*}
Here we have used a change of variables $y = \la_n x$ for the second equality,  and $\| v_n \|^2_{L^2(\re)} \le 1$ for the first inequality.
Letting $n \to \infty$, we have the desired result. 
\qed

\vspace{1em}\noindent
{\it Proof of Theorem \ref{Theorem:asymptotic}}:  

We need to prove that there exists $C_1 > 0$ such that for any $\alpha < \pi$ which is sufficiently close to $\pi$,
it holds that
\[
	A(\alpha) \ge \frac{C_1}{1 - \alpha / \pi}.
\]
Again we use the Moser sequence (\ref{Moser sequence}) and we test $A(\alpha)$ by $v_{\eps} = u_{\eps} /\| (-\Delta)^{1/4} u_{\eps} \|_{L^2(\re)}$.
As in the similar calculations in the proof of Theorem \ref{Theorem:fractional AT}, we have
\begin{align*}
	A(\alpha) &\ge \frac{1}{\| v_{\eps} \|_{L^2(\re)}^2} \int_{\re} \( e^{\alpha v_{\eps}^2} - 1 \) dx \\
	&\ge \frac{(1/2)}{\| v_{\eps} \|_{L^2(\re)}^2} \int_{|x| \le \eps} e^{\alpha v_{\eps}^2} dx \\
	&\ge C \eps \( \log (1/\eps) \) \exp \( \frac{\alpha}{\pi} \frac{\log (1/\eps)}{1 + (C \log (1/\eps))^{-1}} \) \\
	&= C \eps \( \log (1/\eps) \) \exp \( \delta_{\eps} \log (1/\eps) \)
\end{align*}
where we put $\delta_{\eps} =  (\frac{\alpha}{\pi}) \frac{1}{1 + (C \log (1/\eps))^{-1}} \in (0,1)$.

Now, for $\alpha < \pi$ which is sufficiently close to $\pi$, we fix $\eps > 0$ small such that
\begin{equation}
\label{log eps}
	\frac{1}{1 - \alpha / \pi} \le \log (1/\eps) \le \frac{2}{1 - \alpha / \pi},
\end{equation}
which implies
\[
	\exp \( -\frac{2}{1-\alpha/\pi} \) \le \eps \le \exp \( -\frac{1}{1-\alpha/\pi} \).
\]
With this choice of $\eps >0$, we have
\begin{align}
	A(\alpha) &\ge C \eps \( \log (1/\eps) \) \exp \( \delta_{\eps} \log (1/\eps) \) \notag \\
\label{A lower}
	&= C \eps \(\log (1/\eps)\) (1/\eps)^{\delta_{\eps}} = C \eps^{1 - \delta_{\eps}} \(\log (1/\eps)\).
\end{align}
Now, we estimate that
\begin{align*}
	\eps^{1 - \delta_{\eps}} &\ge \( \exp \( -\frac{2}{1-\alpha/\pi} \) \)^{1-\delta_{\eps}} 
	= \exp \( -\frac{2}{1-\alpha/\pi} (1-\delta_{\eps}) \) \\ 
	&= \exp \( -\(\frac{2}{1-\alpha/\pi}\) \left\{ (1 - \alpha/\pi) + (\alpha/\pi) \( 1 - \frac{1}{1 + (C \log 1/\eps )^{-1}} \) \right\} \) \\ 
	&= \exp \( -2 - \(\frac{2(\alpha/\pi)}{1-\alpha/\pi}\) \( \frac{1}{1 + C \log 1/\eps} \) \) \\ 
	&\ge \exp \( -2 - \(\frac{2(\alpha/\pi)}{1-\alpha/\pi}\) \( \frac{1}{1 + \frac{C}{1 - \alpha/\pi}} \) \) \\ 
	&= e^{-2} \cdot e^{- \frac{2(\alpha/\pi)}{C + 1 - \alpha/\pi}} = e^{-2} \cdot e^{-f(\alpha/\pi)}
\end{align*}
where $f(t) = \frac{2t}{C + 1 - t}$ for $t \in [0,1]$ and we have used (\ref{log eps}) in the last inequality.
We easily see that $f(0) = 0$, $f^{\prime}(t) = \frac{2(C+1)}{(C + 1 - t)^2} > 0$ for $t > 0$, thus $f(t)$ is strictly increasing in $t$ and
$\max_{t \in [0,1]} f(t) = f(1) = 2/C$.
Thus we have
\begin{align*}
	\eps^{1 - \delta_{\eps}} \ge e^{-2} \cdot e^{-2/C} =: C_0
\end{align*}
which is independent of $\alpha$.
Backing to (\ref{A lower}) with (\ref{log eps}), we observe that
\begin{align*}
	A(\alpha) \ge C \eps^{1 - \delta_{\eps}} \(\log (1/\eps) \) \ge C C_0 \(\log (1/\eps) \) \ge \frac{C C_0}{1 - \alpha/\pi}
\end{align*}
which proves the result.
\qed

\section{Proof of Theorem \ref{Theorem:Attainability} and \ref{Theorem:Non-attainability}}

For $u \in H^{1/2,2}(\re)$, $u^{*}$ will denote its symmetric decreasing rearrangement defined as follows:
For a measurable set $A \subset \re$, let $A^{*}$ denote an open interval $A^{*} = (-|A|/2, |A|/2)$.
We define $u^{*}$ by
\[
	u^{*}(x) = \int_0^{\infty} \chi_{\{ y \in \re : |u(y)| > t \}^{*}}(x) dt
\]
where $\chi_A$ denote the indicator function of a measurable set $A \subset \re$.
Note that $u^*$ is nonnegative, even, and decreasing on the positive line $\re_+ = [0, +\infty)$.
It is known that
\begin{equation}
	\int_{\re} F(u^*) dx = \int_{\re} F(|u|) dx
\end{equation}
for any nonnegative measurable function $F: \re_{+} \to \re_{+}$, which is the difference of two monotone increasing functions $F_1, F_2$ with $F_1(0) = F_2(0) = 0$
such that either $F_1 \circ |u|$ or $F_2 \circ |u|$ is integrable.
Also the inequality of P\'olya-Szeg\"o type
\[
	\int_{\re} |(-\Delta u^*)^{1/4}|^2 dx \le \int_{\re} |(-\Delta u)^{1/4}|^2 dx
\]
holds true for $u \in H^{1/2,2}(\re)$, see for example, \cite{Almgren-Lieb} and \cite{Lieb-Loss}.


\begin{remark}
Note that Radial Compactness Lemma by Strauss \cite{Strauss} is violated on $\re$.
More precisely, let
\[
	H^{1/2,2}_{rad}(\re) = \{ u \in H^{1/2,2}(\re) \, : \, u(x) = u(-x), \, x \ge 0 \},
\]
then $H^{1/2,2}_{rad}(\re)$ cannot be embedded compactly in $L^q(\re)$ for any $q > 0$. 
To see this, let $\psi \ne 0$ be an even function in $C_c^{\infty}(\re)$ with $\text{supp}(\psi) \subset (-1,1)$
and put $u_n(x) = \psi(x-n) + \psi(x+n)$. 
Then we see $u_n$ is even, compactly supported smooth function, and $u_n \weakto 0$ weakly in $H^{1/2,2}(\re)$ as $n \to \infty$. 
But $\{ u_n \}$ does not have any strong convergent subsequence in $L^q(\re)$,
because $\| u_n \|_{L^q(\re)}^q = 2 \| \psi \|_{L^q(\re)}^q > 0$ for any $n$ sufficient large.
\end{remark}

However, for a sequence $\{ u_n \}_{n \in \N} \subset H^{1/2,2}(\re)$ with $u_n$ even, nonnegative and decreasing on $\re_{+}$,
we have the following compactness result.

\begin{prop}
\label{prop:compactness}
Assume $\{ u_n \} \subset H^{1/2,2}(\re)$ be a sequence such that $u_n$ is even, nonnegative and decreasing on $\re_{+}$. 
Let $u_n \weakto u$ weakly in $H^{1/2,2}(\re)$.
Then
$u_n \to u$ strongly in $L^q(\re)$ for any $q \in (2, +\infty)$ for a subsequence. 
\end{prop}

\begin{proof}
Since $\{ u_n \} \subset H^{1/2,2}(\re)$ is a weakly convergent sequence, we have $\sup_{n \in \N} \| u_n \|_{H^{1/2,2}(\re)} \le C$ for some $C > 0$.
We also have $u_n(x) \to u(x)$ a.e $x \in \re$ for a subsequence, thus $u$ is even, nonnegative and decreasing on $\re_{+}$. 
Now, we use the estimate below, which is referred to a Simple Radial Lemma:
If $u \in L^2(\re)$ is even, nonnegative and decreasing on $\re_+$, then it holds
\begin{equation}
\label{Radial Lemma}
	u^2(x) \le \frac{1}{2|x|} \int_{-|x|}^{|x|} u^2(y) dy \le \frac{1}{2|x|}\| u \|^2_{L^2(\re)} \quad (x \ne 0).
\end{equation}
Thus $u_n^2(x)\le \frac{C}{2|x|}$ for $x \ne 0$ by $\sup_{n \in \N} \| u_n \|_{H^{1/2,2}(\re)} \le C$ 
and $u^2(x)\le \frac{C}{2|x|}$ by the pointwise convergence.
Now, set $v_n = | u_n - u|^q$ for $q > 2$. 
Then we see $v_n(x) \to 0$ a.e. $x \in \re$.
Moreover,
\begin{align*}
	\int_{|x| \ge R} |u_n - u|^q dx &= 2 \int_R^{\infty} |u_n - u|^q dx \\ 
	&\le 2^q \(\int_R^{\infty} |u_n|^q dx + \int_R^{\infty} |u|^q dx \) \\
	&\le C \int_R^{\infty} \frac{dx}{| x |^{q/2}} = \frac{C R^{1-q/2}}{(q/2) - 1} \to 0
\end{align*}
as $R \to \infty$ since $q > 2$.
Thus $\{ v_n \}_{n \in \N}$ is uniformly integrable.
Also by \cite{Hitchhiker's guide} Theorem 6.9, we know that
\[
	H^{1/2,2}(\re) \subset L^{q_0}(\re) \quad \text{for any $q_0 \ge 2$ and} \quad \| u \|_{L^{q_0}(\re)} \le C \| u \|_{H^{1/2,2}(\re)}.
\]
For any $q > 2$, take $q_0$ such that $2 < q < q_0 < \infty$.
Since $u_n - u$ is uniformly bounded in $H^{1/2,2}(\re)$, we have $\| u_n - u \|_{L^{q_0}(\re)} \le C$, and
\[
	\int_I v_n dx = \int_I |u_n - u|^q dx \le \( \int_I |u_n - u|^{q_0} dx \)^{q/q_0} |I|^{1-q/q_0}
\]
for any bounded measurable set $I \subset \re$.
Therefore $\int_I v_n dx \to 0$ if $|I| \to 0$, which implies $\{ v_n \}$ is uniformly absolutely continuous.
Thus by Vitali's Convergence Theorem (see for example, \cite{Folland} p.187), we obtain $v_n = |u_n - u|^q \to 0$ strongly in $L^1(\re)$, which is the desired conclusion.
\end{proof}

\begin{prop}
\label{Porp:Ishiwata-Nakamura-Wadade:Lemma3.1}
Assume $\{ u_n \} \subset H^{1/2,2}(\re)$ be a sequence with $\| (-\Delta)^{1/4} u_n \|_{L^2(\re)} \le 1$.  
Let $u_n \weakto u$ weakly in $H^{1/2,2}(\re)$ for some $u$ and assume $u_n$ is even, nonnegative and decreasing on $\re_{+}$. 
Then we have
\begin{align*}
	\int_{\re} \( e^{\alpha u_n^2} -1 - \alpha u_n^2 \) dx \to \int_{\re} \( e^{\alpha u^2} -1 - \alpha u^2 \) dx
\end{align*}
for any $\alpha \in (0, \pi)$.
\end{prop}

\begin{proof}
The similar  proposition above is already appeared, see \cite{Ishiwata-Nakamura-Wadade} Lemma 3.1, and \cite{Dong-Lu} Lemma 5.5.
We prove it here for the reader's convenience.

Put $\Pha(t) = e^{\alpha t^2} -1$ and $\Psa(t) = e^{\alpha t^2} -1 - \alpha t^2$. 
Note that $\Phi_{\alpha}(t)$ is nonnegative, strictly convex 
and $\Psi_{\alpha}^{\prime}(t) = 2\alpha t \Phi_{\alpha}(t)$.
Thus by the mean value theorem, we have
\begin{align*}
	|\Psa(u_n) - \Psa(u)| &\le \Psa^{\prime}(\theta u_n + (1-\theta) u) |u_n - u| \\ 
	&\le 2\alpha |\theta u_n + (1-\theta) u| \Pha(\theta u_n + (1-\theta) u) |u_n - u| \\
	&\le 2\alpha (|u_n| + |u|) \( \theta \Pha(u_n) + (1-\theta) \Pha(u) \) |u_n - u| \\
	&\le 2\alpha (|u_n| + |u|) \( \Pha(u_n) + \Pha(u) \) |u_n - u|.
\end{align*}
Thus we have
\begin{align}
\label{P1}
	&\int_{\re} |\Psa(u_n) - \Psa(u)| dx \le 2\alpha \int_{\re} (|u_n| + |u|) \( \Pha(u_n) + \Pha(u) \) |u_n - u| dx \notag \\
	&\le 2\alpha \| |u_n| + |u| \|_{L^a(\re)} \| \Pha(u_n) + \Pha(u) \|_{L^b(\re)} \| u_n - u \|_{L^c(\re)}
\end{align}
by H\"older's inequality, where $a,b,c > 1$ and $1/a + 1/b + 1/c = 1$ are chosen later.

First, direct calculation shows that
\begin{equation}
\label{Pha^b}
	\( \Pha(t) \)^b < e^{b\alpha t^2} - 1 \quad (t \in \re)
\end{equation}
for all $b > 1$.
Thus if we fix $1 < b < \pi/\alpha$ so that $b \alpha < \pi$ is realized, then we have
\begin{align*}
	&\| \Pha(u_n) + \Pha(u) \|^b_{L^b(\re)} \le \( \| \Pha(u_n) \|_{L^b(\re)} + \| \Pha(u) \|_{L^b(\re)} \)^b \\
	&\le 2^{b-1} \( \int_{\re} \( \Pha(u_n) \)^b dx + \int_{\re}  \(\Pha(u) \)^b dx \) \\
	&\le 2^{b-1} \( \int_{\re} \( e^{b\alpha u_n^2} - 1 \) dx + \int_{\re} \( e^{b\alpha u^2} - 1 \) dx \) \\
	&\le 2^{b-1} A(b \alpha) \( \| u_n \|_{L^2(\re)}^2 + \| u \|_{L^2(\re)}^2 \),
\end{align*}
here we used (\ref{Pha^b}) for the third inequality and Theorem \ref{Theorem:fractional AT} for the last inequality,
the use of which is valid since $\| (-\Delta)^{1/4} u_n \|_{L^2(\re)} \le 1$ and $\| (-\Delta)^{1/4} u \|_{L^2(\re)} \le 1$ by the weak lower semicontinuity.
Note that $\{ u_n \}$ satisfies $\sup_{n \in \N} \| u_n \|_{H^{1/2,2}(\re)} \le C$ for some $C > 0$.
Thus we have obtained $\| \Pha(u_n) + \Pha(u) \|_{L^b(\re)} = O(1)$ independent of $n$.

Next, we estimate the term $\| |u_n| + |u| \|_{L^a(\re)}$.
Since $\{ u_n \}$ is a bounded sequence in $H^{1/2,2}(\re)$, we have by \cite{Hitchhiker's guide} Theorem 6.9 that
$\| u \|_{L^q(\re)} \le C \| u_n \|_{H^{1/2,2}(\re)}$ for any $q \ge 2$.
Thus we see $\| |u_n| + |u| \|_{L^a(\re)} \le C$ for some $C > 0$ independent of $n$ for $a \ge 2$.
Now, note that if we choose $1 < b < \pi/\alpha$ and $a > 2$ sufficiently large, 
then we can find $c > 2$ such that $1/a + 1/b + 1/c = 1$.

By these choices and Proposition \ref{prop:compactness}, we conclude that
$\| u_n - u \|_{L^c(\re)} \to 0$ as $n \to \infty$.
Backing to (\ref{P1}) with all together, we conclude that
\[
	\int_{\re} \Psa(u_n) dx \to \int_{\re} \Psa(u) dx \quad (n \to \infty),
\]
which is the desired conclusion.
\end{proof}

Now, we prove Theorem \ref{Theorem:Attainability}.
We will show that $A(\alpha)$ in (\ref{Aalpha}) is attained for any $0 < \alpha < \pi$.
Since $A(\alpha) = \tilde{A}(\alpha)$ by Lemma \ref{Lemma1}, we choose a maximizing sequence for $\tilde{A}(\alpha)$:
\[
	\int_{\re} \( e^{\alpha u_n^2} -1 \) dx = A(\alpha) + o(1) \quad (n \to \infty).
\]
Here $\{ u_n \}_{n \in \N} \subset H^{1/2,2}(\re)$ satisfies $\| (-\Delta)^{1/4}u_n \|_{L^2(\re)} \le 1$ and $\| u_n \|_{L^2(\re)} = 1$.
By appealing to the use of rearrangement, we may furthermore assume that $u_n$ is nonnegative, even, and decreasing on $\re_{+}$.
Since $\{ u_n \}_{n \in \N} \subset H^{1/2,2}(\re)$ is a bounded sequence, we have $u \in H^{1/2,2}(\re)$ such that $u_n \weakto u$ in $H^{1/2,2}(\re)$.
By Proposition \ref{Porp:Ishiwata-Nakamura-Wadade:Lemma3.1}, we see
\[
	\int_{\re} \( e^{\alpha u_n^2} -1 -\alpha u_n^2 \) dx = \int_{\re} \( e^{\alpha u^2} -1 -\alpha u^2 \) dx
\]
as $n \to \infty$.
Therefore, since $\| u_n \|_{L^2(\re)}^2 = 1$, we have, letting $n \to \infty$,
\begin{equation}
\label{A1}
	A(\alpha) = \alpha + \int_{\re} \( e^{\alpha u^2} -1 -\alpha u^2 \) dx.
\end{equation}

Next, we claim that $A(\alpha) > \alpha$ for any $0 < \alpha < \pi$.
Indeed, 
take any $u_0\in H^{1/2,2}(\re)$ such that $u_0 \not\equiv 0$, $\| (-\Delta)^{1/4}u_0 \|_{L^2(\re)} \le 1$ and $\| u_0 \|_{L^2(\re)} = 1$.
Then we have
\[
	A(\alpha) = \tilde{A}(\alpha) \ge \int_{\re} \( e^{\alpha u_0^2} -1 \) dx = \alpha + \int_{\re} \( e^{\alpha u_0^2} -1 -\alpha u_0^2 \) dx.
\]
Now, 
since $e^{\alpha t^2} - 1 - \alpha t^2 > 0$ for any $t > 0$, we have 
\[
	\int_{\re} \( e^{\alpha u_0^2} -1 -\alpha u_0^2 \) dx > 0
\]
for $u_0 \not\equiv 0$, which results in $A(\alpha) > \alpha$, the claim.

By the claim and (\ref{A1}), we conclude that the weak limit $u$ satisfies $u \not\equiv 0$.
By the weak lower semi continuity, we have $u \not\equiv 0$ satisfies
$\| (-\Delta)^{1/4}u \|_{L^2(\re)} \le 1$ and $\| u \|_{L^2(\re)} \le 1$.
Thus by (\ref{A1}) again, we see
\begin{align*}
	A(\alpha) &= \alpha + \int_{\re} \( e^{\alpha u^2} -1 -\alpha u^2 \) dx \\
	&\le \alpha + \frac{1}{\| u \|_{L^2(\re)}^2} \int_{\re} \( e^{\alpha u^2} -1 -\alpha u^2 \) dx \\
	&= \alpha + \frac{1}{\| u \|_{L^2(\re)}^2} \int_{\re} \( e^{\alpha u^2} -1 \) dx -\alpha \frac{\| u \|_{L^2(\re)}^2}{\| u \|_{L^2(\re)}^2} \\
	&= \frac{1}{\| u \|_{L^2(\re)}^2} \int_{\re} \( e^{\alpha u^2} -1 \) dx.
\end{align*}
Thus we have shown that $u \in H^{1/2,2}(\re)$ maximizes $A(\alpha)$.
\qed

Next, we prove Theorem \ref{Theorem:Non-attainability}. 
We follow Ishiwata's argument in \cite{Ishiwata}.
Let
\begin{align*}
	&M = \left\{ u \in H^{1/2,2}(\re) \, : \, \| u \|_{H^{1/2,2}(\re)} = 1 \right\}, \\
	&J_{\alpha} :M \to \re, \quad J_{\alpha}(u) = \int_{\re} \( e^{\alpha u^2} - 1 \) dx.
\end{align*}
Actually, we will show a stronger claim that $J_{\alpha}$ has no critical point on $M$ for sufficiently small $\alpha > 0$.
Assume the contrary that there exists a critical point $v \in M$ of $J_{\alpha}$ for small $\alpha >0$.
Then we define an orbit on $M$ through $v$ as
\[
	v_{\tau}(x) = \sqrt{\tau} v(\tau x) \quad \tau \in (0,\infty), \quad w_{\tau} = \frac{v_{\tau}}{\| v_{\tau} \|_{H^{1/2}}} \in M.
\]
Note that $w_1 = v$ thus it must be $\frac{d}{d\tau} \Big|_{\tau = 1} J_{\alpha}(w_{\tau}) = 0$.
By scaling rules (\ref{scaling}), we see for any $p \ge 2$,
\begin{align*}
	\| v_{\tau} \|^p_{L^p(\re)} = \tau^{p/2-1}\| v \|^p_{L^p(\re)} \quad \text{and} \quad \| (-\Delta)^{1/4} v_{\tau} \|_{L^2(\re)} = \tau \| (-\Delta)^{1/4} v \|_{L^2(\re)}.
\end{align*}
Now, we see
\begin{align*}
	&J_{\alpha}(w_{\tau}) = \int_{\re} \( e^{\alpha w_{\tau}^2} - 1 \) dx
	= \int_{\re} \sum_{j=1}^{\infty} \frac{\alpha^j}{j!} \frac{v_{\tau}^{2j}(x)}{\( \| v_{\tau} \|_2^2 + \| (-\Delta)^{1/4} v_{\tau} \|_2^2 \)^j} \\
	&= \sum_{j=1}^{\infty} \frac{\alpha^j}{j!} \frac{\| v_{\tau} \|_{2j}^{2j}}{\( \| v_{\tau} \|_2^2 + \| (-\Delta)^{1/4} v_{\tau} \|_2^2 \)^j}
	= \sum_{j=1}^{\infty} \frac{\alpha^j}{j!} \frac{\tau^{j-1} \| v \|_{2j}^{2j}}{\( \| v \|_2^2 + \tau \| (-\Delta)^{1/4} v \|_2^2 \)^j} \\
	&= \sum_{j=1}^{\infty} \frac{\alpha^j}{j!} f_j(\tau)
\end{align*}
where $f_j(\tau) = \frac{\tau^{j-1} c}{(b + \tau a)^j}$ with
$a = \| (-\Delta)^{1/4} v \|_2^2$, $b =  \| v \|_2^2$ and $c =  \| v \|_{2j}^{2j}$.
Since
\[
	f_j'(\tau) = \frac{\tau^{j-2} c}{(b + \tau a)^{j+1}} \left\{ -\tau a + (j-1) b \right\}
\]
and $\| (-\Delta)^{1/4} v \|_2^2 + \| v \|_2^2 = 1$,
we calculate
\begin{align*}
	&\frac{d}{d\tau} \Big|_{\tau = 1} J_{\alpha}(w_{\tau}) \\ 
	&= \sum_{j=1}^{\infty} \left[ \frac{\alpha^j}{j!} \frac{\tau^{j-2} \| v \|_{2j}^{2j}}{\( \| v \|_2^2 + \tau \| (-\Delta)^{1/4} v \|_2^2 \)^{j+1}} 
	\left\{ -\tau \| (-\Delta)^{1/4} v \|_2^2 + (j-1) \| v \|_2^2 \right\} \right]_{\tau = 1} \\
	&\le -\alpha \| (-\Delta)^{1/4} v \|_2^2 \| v \|_2^2 + \sum_{j=2}^{\infty} \frac{\alpha^{j}}{(j-1)!} \| v \|_{2j}^{2j} \\
	&= \alpha \| (-\Delta)^{1/4} v \|_2^2 \| v \|_2^2 
	\left\{ -1 + \sum_{j=2}^{\infty} \frac{\alpha^{j-1}}{(j-1)!} \frac{\| v \|_{2j}^{2j}}{\| (-\Delta)^{1/4} v \|_2^2 \| v \|_2^2} \right\} .
\end{align*}
Here, we need the following lemma:
\begin{lemma}{(Ogawa-Ozawa \cite{Ogawa-Ozawa})}
\label{Lemma:Ogawa-Ozawa}
There exists $C >0$ such that for any $u \in H^{1/2, 2}(\re)$ and $p \ge 2$, 
it holds
\[
	\| u \|_{L^p(\re)}^p \le C p^{p/2} \| (-\Delta)^{1/4} u \|_{L^2(\re)}^{p-2} \| u \|_{L^2(\re)}^2.
\]
\end{lemma}
For $p = 2j$, Lemma \ref{Lemma:Ogawa-Ozawa} implies 
\begin{align*}
	\frac{\| v \|_{2j}^{2j}}{\| (-\Delta)^{1/4} v \|_2^2 \| v \|_2^2} \le C (2j)^j \underbrace{\| (-\Delta)^{1/4} v \|_2^{2j-4}}_{\le 1 \, (j \ge 2)} \le C (2j)^j.
\end{align*}
Thus for $0 < \alpha << 1$ sufficiently small (it would be enough that $\alpha < 1/(2e)$), 
Stirling's formula $j! \sim j^j e^{-j} \sqrt{2\pi j}$ implies that 
\[
	\sum_{j=2}^{\infty} \frac{\alpha^{j-1}}{(j-1)!} \frac{\| v \|_{2j}^{2j}}{\| (-\Delta)^{1/4} v \|_2^2 \| v \|_2^2}  \le \sum_{j=2}^{\infty} \frac{\alpha^{j-1}}{(j-1)!} (2j)^j \le \alpha C
\]
for some $C > 0$ independent of $\alpha$.
Therefore we have
$\frac{d}{d\tau} J_{\alpha}(w_{\tau}) \Big |_{\tau = 1} < 0$ for small $\alpha$, which is a desired contradiction.
\qed

\section{Proof of Theorem \ref{Theorem:Dong-Lu}.}

In order to prove Theorem \ref{Theorem:Dong-Lu}, first we set
\begin{equation} 
\label{Fbeta}
	F(\beta) = \sup_{u \in H^{1/2,2}(\re) \atop \| u \|_{H^{1/2,2}(\re)} \le 1} \int_{\re} u^2 e^{\beta u^2} dx
\end{equation}
for $\beta > 0$.
Then we have
\begin{prop}
\label{Prop:Fbeta}
We have $F(\beta) < \infty$ for $\beta < \pi$
\end{prop}

\begin{proof}
We follow the proof of Theorem 1.5 in \cite{Iula-Maalaoui-Martinazzi}.
Take any $u \in H^{1/2,2}(\re)$ with $\| u \|_{H^{1/2,2}(\re)} \le 1$ in the admissible sets for $F(\beta)$ in (\ref{Fbeta}). 
By appealing to the rearrangement, we may assume that $u$ is even, nonnegative and decreasing on $\re_{+}$.
We divide the integral
\[
	\int_{\re} u^2 e^{\beta u^2} dx = \int_{\re \setminus I} u^2 e^{\beta u^2} dx + \int_I u^2 e^{\beta u^2} dx = (I) + (II), 
\]
where $I = (-1/2, 1/2)$.

First, we estimate $(I)$.
By the Radial Lemma (\ref{Radial Lemma}), we see for any $k \in \N$, $k \ge 2$,
\[
	u^{2k}(x) \le \( \frac{\| u \|_{L^2(\re)}^2}{2|x|} \)^k = \frac{\| u \|_{L^2(\re)}^{2k}}{2^k} \frac{1}{|x|^k} \quad \text{for} \quad x \ne 0. 
\]
Thus
\begin{align*}
	\int_{\re \setminus I} u^{2k}(x) dx &\le \frac{\| u \|_{L^2(\re)}^{2k}}{2^k} \int_{\re \setminus I} \frac{dx}{|x|^k} \\
	&= \frac{\| u \|_{L^2(\re)}^{2k}}{2^{k-1}} \int_{1/2}^{\infty} \frac{dx}{x^k} = \frac{\| u \|_{L^2(\re)}^{2k}}{k-1}.
\end{align*}
Therefore, we have
\begin{align*}
	(I) &= \int_{\re \setminus I} u^2 e^{\beta u^2} dx = \int_{\re \setminus I} u^2 \( 1 + \sum_{k=1}^{\infty} \frac{\beta^k u^{2k}}{k!} \) dx \\
	&= \int_{\re \setminus I} u^2 dx + \sum_{k=2}^{\infty} \frac{\beta^{k-1}}{(k-1)!} \int_{\re \setminus I} u^{2k} dx \\
	&\le \| u \|_{L^2(\re)}^2 + \sum_{k=2}^{\infty}  \frac{\beta^{k-1}}{(k-1)!}  \frac{\| u \|_{L^2(\re)}^{2k}}{k-1} \\
	&= \| u \|_{L^2(\re)}^2 \( 1  + \sum_{k=2}^{\infty}  \frac{\beta^{k-1}}{(k-1)(k-1)!} \| u \|_{L^2(\re)}^{2(k-1)} \).
\end{align*}
Now by the constraint $\| u \|_{H^{1/2,2}(\re)} \le 1$, we have $\| u \|_{L^2(\re)} \le 1$.
Also if we put $a_k = \frac{\beta^{k-1}}{(k-1)(k-1)!}$, then $\sum_{k=2}^{\infty} a_k$ converges since $a_{k+1}/a_k = \beta \frac{k-1}{k^2} \to 0$ as $k \to \infty$.
Thus we obtain
\[
	(I) \le 1  + \sum_{k=2}^{\infty} \frac{\beta^{k-1}}{(k-1)(k-1)!} \le C
\]
where $C >0$ is independent of $u \in H^{1/2,2}(\re)$ with $\| u \|_{H^{1/2,2}(\re)} \le 1$.

Next, we estimate $(II)$.
Set
\[
	v(x) = \begin{cases}
	u(x) - u(1/2), &\quad |x| \le 1/2, \\
	0, &\quad |x| > 1/2.
	\end{cases}
\]
Then by the argument of \cite{Iula-Maalaoui-Martinazzi}, we know that
\begin{align*}
	&\| (-\Delta)^{1/4} v \|_{L^2(\re)}^2 \le \| (-\Delta)^{1/4} u \|_{L^2(\re)}^2, \\
	&u^2(x) \le v^2(x) \( 1 + \| u \|_{L^2(\re)}^2 \) + 2
\end{align*}
for $x \in I$.
Put $w = v \sqrt{1 +\| u \|_{L^2(\re)}^2}$. 
Then we have $w \in \tilde{H}^{1/2,2}(I)$ since $v \equiv 0$ on $\re \setminus I$, and
\begin{align*}
	&\| (-\Delta)^{1/4} w \|_{L^2(\re)}^2 = \( 1 + \| u \|_{L^2(\re)}^2 \) \| (-\Delta)^{1/4} v \|_{L^2(\re)}^2 \\
	&\le \( 1 + \| u \|_{L^2(\re)}^2 \)\( 1 - \| u \|_{L^2(\re)}^2 \) \le 1.
\end{align*}
Thus we may use the fractional Trudinger-Moser inequality (Proposition \ref{Prop:fractional TM}) to $w$ to obtain
\begin{align*}
	\int_I e^{\pi w^2} dx \le C
\end{align*}
for some $C > 0$ independent of $u$.
By $u^2 \le w^2 + 2$ on $I$, we conclude that
\[
	\int_I e^{\pi u^2} dx \le  \int_I e^{\pi(w^2 + 2)} dx = e^{2\pi} \int_I e^{\pi w^2} dx \le C'.
\]
Now, since $\beta < \pi$, there is an absolute constant $C_0$ such that $t^2 e^{\beta t^2} \le C_0 e^{\pi t^2}$ for any $t \in \re$.
Finally, we obtain
\[
	(II) = \int_I u^2 e^{\beta u^2} dx \le C_0 \int_I e^{\pi u^2} dx \le C_0 C'.
\] 
Proposition \ref{Prop:Fbeta} follows from the estimates $(I)$ and $(II)$.
\end{proof}

By using Proposition \ref{Prop:Fbeta} and arguing as in the proof of Theorem \ref{Theorem:fractional AT}
(after establishing the similar claims as in Lemma \ref{Lemma1} and Lemma \ref{Lemma2}),
it is easy to obtain the following Proposition: 
\begin{prop}
For any $0 < \alpha < \beta < \pi$, we have
\[
	E(\alpha) \le \( \frac{1}{1 - \alpha/\beta} \) F(\beta).
\]
\end{prop}
Since $F(\beta) < \infty$ for any $\beta < \pi$, this proves the first part of Theorem \ref{Theorem:Dong-Lu}.
For the attainability of $E(\alpha)$ for $\alpha \in (0,\pi)$, it is enough to argue as in the proof of Theorem \ref{Theorem:Attainability}.
We omit the details.
\qed

%
%

\vspace{1em}\noindent
{\bf Acknowledgments.}

Part of this work was supported by 
JSPS Grant-in-Aid for Scientific Research (B), No.15H03631, 
JSPS Grant-in-Aid for Challenging Exploratory Research, No.26610030.

\end{document}